\documentclass{article}
\usepackage{graphicx}
\usepackage{authblk}
\usepackage{cite}

\usepackage{tikz}
\usetikzlibrary{arrows,shapes,matrix,
                decorations.pathmorphing,
                shapes.geometric,calc}

\usepackage{graphicx}
\usepackage{float}

\usepackage{xcolor}

\usepackage{amsmath}
\usepackage{amssymb}
\usepackage{amsthm}

\newtheorem{theorem}{Theorem}
\newtheorem{lemma}[theorem]{Lemma}
\newtheorem{corollary}[theorem]{Corollary}
\newtheorem{proposition}[theorem]{Proposition}

\tikzstyle{arc}=[->,shorten <=3pt, shorten >=3pt,
                 >=stealth, line width=1.1pt]
\tikzstyle{edge}=[shorten <=2pt, shorten >=2pt,
                  >=stealth, line width=1.1pt]
\tikzstyle{blackvertex}=[circle, fill=black,
                         minimum size=5pt,
                         inner sep=0pt, outer sep=0pt]
\tikzstyle{vertex}=[circle, fill=white, draw,
                    minimum size=5pt,
                    inner sep=0pt, outer sep=0pt]

\title{Strongly chordal digraphs and $\Gamma$-free matrices
\thanks{The authors gratefully acknowledge support from NSERC Canada}}

\author[1]{Pavol~Hell\thanks{pavol@sfu.ca}}
\author[2]{C\'esar~Hern\'andez-Cruz\thanks{cesar@cs.cinvestav.mx}}
\author[3]{Jing~Huang\thanks{huangj@uvic.ca}}
\author[4]{Jephian~C.-H.~Lin\thanks{chlin@math.nsysu.edu.tw}}

\affil[1]{School of Computing Science,
          Simon Fraser University}
\affil[2]{Departamento de Computaci\'on,
          Centro de Investigaci\'on y de
					Estudios Avanzados del IPN}
\affil[3]{Department of Mathematics and Statistics,
		      University of Victoria}
\affil[4]{Department of Applied Mathematics,
		      National Sun Yat-sen University}

\begin{document}
\date{}

\maketitle
\begin{abstract}

We define strongly chordal digraphs, which generalize strongly chordal graphs, and chordal 
bipartite graphs, and are included in the class of chordal digraphs. They correspond to square 
$0, 1$ matrices that admit a simultaneous row and column permutation avoiding the $\Gamma$
matrix. In general, it is not clear if these digraphs can be recognized in polynomial time, and we 
focus on symmetric digraphs (i.e., graphs with possible loops), tournaments with possible loops, 
and balanced digraphs. In each of these cases we give a polynomial-time recognition algorithm 
and a forbidden induced subgraph characterization. We also discuss an algorithm for minimum
general dominating set in strongly chordal graphs with possible loops, extending and unifying
similar algorithms for strongly chordal graphs and chordal bipartite graphs.
\end{abstract}

\section{Background and definitions}

A number of interesting graph classes have been extended to digraphs, including interval graphs
\cite{federDAM160}, chordal graphs \cite{haskinsSIAMJC2, meisterTCS463,hellDAM216}, split graphs
\cite{hellDAM216,lamarDM312}, and graphs of bounded treewidth \cite{hunterTCS399, johnsonJCTB82}. 
In most cases, there is more than one way to define such a generalization, and it is not obvious which 
one best captures the analogy to the undirected case. (In the undirected case there may be several
equivalent characterizations of the graphs in the class, and each may suggest a different generalization,
which are not equivalent in the context of digraphs.) It seems to be the case that often the most successful
generalizations use the ordering characterization of the undirected concept, or, equivalently, its 
characterization by forbidden submatrices of the adjacency matrix.

Consider first the undirected notion of an interval graph. Since every interval intersects itself, we will assume 
each vertex has a loop. Then interval graphs are known to have the following ordering characterization 
\cite{federDAM160}. (There are other ordering characterizations of interval graphs, but this one turns out to be
most useful; however, it only applies if every vertex is considered adjacent to itself.) A graph $G$ is an interval 
graph if and only if its vertices can be ordered as $v_1, v_2, \dots, v_n$ so that if $i < j$ and $k < \ell$, not
necessarily all distinct, then for $v_iv_{\ell} \in E(G), v_jv_k \in E(G)$ we also have $v_jv_{\ell} \in E(G)$. 
Equivalently, $G$ is an interval graph if and only if the rows and columns of its adjacency matrix can be 
simultaneously permuted to avoid a submatrix of the form
$\left[ \begin{matrix}
  \ast & 1 \\
  1 & 0
\end{matrix} \right]$
where $\ast$ can be either $0$ or $1$.

In \cite{federDAM160}, the authors analogously define a digraph analogue of interval graphs as follows. 
A digraph with a loop 
at every vertex is an {\em adjusted interval digraph} if the rows and columns of its adjacency matrix can be 
simultaneously permuted to avoid a submatrix of the form
$\left[ \begin{matrix}
  \ast & 1 \\
  1 & 0
\end{matrix} \right].$

It turns out that these digraphs have a natural geometric representation, a forbidden structure characterization, 
and other desirable properties analogous to interval graphs \cite{federDAM160}. (By contrast, the earlier class of 
{\em interval digraphs} \cite{dasJGT13}, based on a simple geometric analogy, lacks many of these nice properties.)

To simplify the language, we will say that a vertex is {\em reflexive} if it has a loop and {\em irreflexive} if it does not. 
A digraph is {\em reflexive} if every vertex is reflexive and is {\em irreflexive} if every vertex is irreflexive. Thus the 
diagonal entries of the adjacency matrix of a reflexive digraph are all $1$ and of an irreflexive digraph are all $0$.
An arc $uv$ in a digraph is {\em symmetric} if $vu$ is also an arc. A digraph is {\em symmetric} if every arc is 
symmetric. The adjacency matrix of a symmetric digraph is symmetric. A symmetric digraph may be viewed as 
a graph with possible loops. In the figures, we will depict reflexive vertices in {\em black} and irreflexive vertices 
in {\em white}.

For graph classes that are characterized as intersection graphs (typically chordal graphs and their subclasses
such as strongly chordal graphs and interval graphs), it is most natural to restrict attention to reflexive graphs
(and digraphs), as is noted above for interval graphs. Nevertheless, it is possible to obtain useful generalizations
for digraphs that are neither reflexive nor irreflexive. This is done, for example, in \cite{arash,ross}, where general
digraphs (that have some vertices with loops and others without) avoiding
$\left[ \begin{matrix}
  * & 1 \\
  1 & 0
\end{matrix} \right]$
are investigated and found a useful unification of interval graphs, adjusted interval digraphs, two-dimensional 
orthogonal ray graphs (alias interval containment digraphs), and complements of threshold tolerance graphs.
Another situation where it is fruitful to admit some vertices with loops and others without loops is the subject
of the next section; the class of graphs investigated there unifies reflexive strongly chordal graphs and irreflexive
chordal bigraphs, and introduces a whole new class of well structured graphs.

In this paper we consider the digraph generalization of the undirected notion of strong chordality. A chordal  
graph $G$ can be defined by the existence of a {\em perfect elimination ordering}, also known as a {\em simplicial ordering}, 
$v_1, v_2, \dots, v_n$ of its vertices so that if $i < j, i < k$ and $v_iv_j \in E(G), v_iv_k \in E(G)$, then we must also have 
$v_jv_k \in E(G)$. They are also characterized as those graphs that have no induced cycle of length greater
than three, or those graphs that are intersection graphs of subtrees of a tree \cite{golumbic}. As noted above, we 
consider chordal graphs to be reflexive, i.e., the adjacency matrix of a chordal graph has $1$'s on its main diagonal. 
Then a perfect elimination ordering corresponds to a simultaneous permutation of the rows and columns of the adjacency 
matrix that avoids as a principal submatrix the so-called $\Gamma$
{\em matrix}
$\left[ \begin{matrix}
  1 & 1 \\
  1 & 0
\end{matrix} \right]$.
Such a submatrix is called a {\em principal submatrix} if the upper left $1$ lies on the main diagonal.

Chordal digraphs were first defined in \cite{haskinsSIAMJC2}, and further studied in \cite{meisterTCS463}. A reflexive
digraph $D$ is a {\em chordal digraph} if the rows and columns of its adjacency matrix can be simultaneously permutated
to avoid $\Gamma$ as a principal submatrix. These digraphs can be recognized in polynomial time \cite{haskinsSIAMJC2}
and structural characterizations are known for several special cases, including oriented graphs and semi-complete
digraphs \cite{meisterTCS463}. A more restrictive notion of {\em strict chordal digraphs} from \cite{hellDAM216} admits a
general forbidden induced subgraph characterization and leads to a nice notion of strict split digraphs \cite{hellDAM216}. 

In the context of undirected graphs, {\em strongly chordal graphs} \cite{farberDM43} are defined as the subclass of those
chordal graphs for which the rows and columns of their adjacency matrix can be simultaneously permutated to avoid 
$\Gamma$ as {\em any} submatrix (not just principal submatrix). Strongly chordal graphs admit elegant forbidden structure
characterizations \cite{farberDM43,cn}, efficient recognition algorithms \cite{lubiwSIAMJC16}, and lead to efficient algorithms 
for some problems that are intractable for chordal graphs \cite{farberDM43}.

Permuting rows and columns of a $0, 1$ matrix $M$ to avoid $\Gamma$ as a submatrix has been much studied 
\cite{ansteeJA5,hoffmanSIAMJADM6, lubiwMSc,lubiwSIAMJC16}. A {\em $\Gamma$-free ordering} of $M$ is a 
matrix obtained from $M$ by {\em independently} permuting its rows and columns, to avoid $\Gamma$ as a submatrix.
If the constraint matrix of a linear program is presented in a $\Gamma$-free ordering, then it can be solved by a
greedy algorithm \cite{ansteeJA5,hoffmanSIAMJADM6}. A {\em cycle matrix} is a square $0, 1$ matrix of size
at least $3$, with exactly two $1$'s in each row and each column. A matrix $M$ is {\em totally balanced}, if it admits 
no cycle matrix as a submatrix.  A matrix $M$ admits a $\Gamma$-free ordering if and only if it is totally balanced  
\cite{hoffmanSIAMJADM6}.  There are efficient algorithms to decide if a matrix is totally balanced 
\cite{lubiwMSc,lubiwSIAMJC16}.

For a square matrix $M$, a {\em symmetric $\Gamma$-free ordering} is a matrix obtained from $M$ by {\em 
simultaneously} permuting its rows and columns, to avoid $\Gamma$ as a submatrix. A reflexive graph $G$ is
strongly chordal if and only if its adjacency matrix $M(G)$ has a symmetric $\Gamma$-free ordering \cite{farberDM43}.
The algorithm in \cite{lubiwSIAMJC16} finds a symmetric $\Gamma$-free ordering of a symmetric matrix $M$ 
(or decides that one doesn't exist) provided $M$ has $1$'s on the main diagonal. In particular, a symmetric
matrix $M$ with $1$'s on the main diagonal admits a symmetric $\Gamma$-free ordering if and only 
if it is totally balanced \cite{farberDM43,lubiwSIAMJC16}.

For a bigraph $G$ (a bipartite graph with a fixed bipartition into red and blue vertices), we consider the {\em 
bi-adjacency matrix} $N(G)$, with rows indexed by the red vertices and columns indexed by the blue vertices,
and $N(i, j)=1$ if and only if the $i$-th red vertex is adjacent to the $j$-th blue vertex. Note that $N$ is in general
not a square matrix. A {\em chordal bigraph} $G$ is a bigraph whose bi-adjacency matrix has a $\Gamma$-free 
ordering \cite{hmp}.

We say $D$ is {\em a strongly chordal digraph} if its adjacency matrix $M(D)$ admits a symmetric $\Gamma$-free 
ordering. It follows that a strongly chordal graph is precisely (the underlying graph of) a strongly chordal digraph 
that is symmetric and reflexive. It also follows that strongly chordal digraphs are chordal digraphs as defined in 
\cite{haskinsSIAMJC2,meisterTCS463}. Chordal bigraphs can also be seen as special strongly chordal digraphs, 
because the adjacency matrix $M(G)$ of a bigraph $G$ (viewed as a graph) has a symmetric $\Gamma$-free 
ordering if and only if its bi-adjacency matrix $N(G)$ has a $\Gamma$-free ordering. Thus strongly chordal 
digraphs can be seen as generalizing strongly chordal graphs, and chordal bigraphs, and be included in the
class of chordal digraphs.

The problem of recognizing strongly chordal digraphs is equivalent to the problem of deciding if a given 
square $0, 1$ matrix has a symmetric $\Gamma$-free ordering. This seems to be a difficult problem;
as we show below, it is no longer equivalent with being totally balanced, or any of the other polynomial
conditions that applied for symmetric matrices with $1$'s on the main diagonal.

We shall focus on certain particular classes of digraphs. The first class is the class of symmetric digraphs, i.e., 
graphs with possible loops. This is a non-trivial extension of the two original concepts of reflexive strongly chordal
graphs and irreflexive chordal bigraphs. While some of the tools used in the classical concept do not apply, 
we still recover a reasonable theory and give a full characterization of these digraphs by forbidden subgraphs. 
We also consider the special case of tournaments with possible loops; here we prove that very few of these
tournaments are strongly chordal, and we can actually describe all strongly chordal cases. We also consider
strongly chordal balanced digraphs, which are a different generalization of chordal bigraphs, and include all
oriented trees.

As an example potential application we define a general domination number, which specializes to the usual
domination number in case of reflexive graphs, and to the total domination number in the case of irreflexive
graphs. We give a linear time algorithm to compute the general domination number for strongly chordal
graphs with possible loops, unifying and extending the algorithms given in \cite{damdom,domdom}.

\section{Graphs with possible loops}

In this section we focus on digraphs that are symmetric, and view them as graphs with possible loops.
This involves treating each symmetric pair of arcs $xy, yx$ as one undirected edge $xy$. (Note that the
adjacency matrix of this object is the same whether it is viewed as a symmetric digraph or a graph with 
possible loops.) We first translate the above definitions into a language more consistent with \cite{farberDM43},
where the case of reflexive strongly chordal graphs was first treated.

Let $G$ be a graph with possible loops. Then $G$ is strongly chordal, i.e., its adjacency matrix $M(G)$ has a 
symmetric $\Gamma$-free ordering, if and only if the vertices of $G$ can be linearly ordered as $v_1, v_2, \dots, v_n$ 
so that if $i < j, k < \ell$ and $v_iv_k \in E(G), v_iv_{\ell} \in E(G), v_jv_k \in E(G)$ (where $i, j, k, \ell$ are not 
necessarily all distinct), then we also have $v_jv_{\ell} \in E(G)$. We call such an ordering a {\em strong ordering} 
of $G$. A vertex $v \in V(G)$ is {\em simple} if its  neighbours have their neighbourhoods linearly ordered by inclusion, 
i.e., if for any $x, y \in N(v)$, we have $N(x) \subseteq N(y)$ or $N(x) \supseteq N(y)$. A {\em simple ordering} of 
$G$ is a linear ordering $v_1, v_2, \dots, v_n$ of $V(G)$ such that each $v_i$ is simple in the subgraph induced 
by the set $\{v_i, v_{i+1}, \dots, v_n\}$. It is easy to see that a strong ordering is a simple ordering. We will prove
that the converse also holds. These notions and facts are analogous to the usual theory for reflexive graphs 
\cite{farberDM43}, except for us the neighbourhood of a vertex may or may not include that vertex, depending 
on whether the vertex is reflexive or not, respectively. A reflexive graph is strongly chordal, i.e., has a strong ordering,
if and only if it has a simple ordering \cite{farberDM43}. A reflexive graph is strongly chordal if and only if it does
not contain an induced cycle of length greater than $3$ or an induced trampoline \cite{farberDM43}. A {\em
trampoline} is a complete graph on $x_1, x_2, \dots, x_k, k \geq 3,$ with vertices $y_1, y_2, \dots, y_k$ each 
of degree $2$, where each $y_i$ is adjacent to $x_{i-1}$ and $x_{i+1}$ (subscripts modulo $k$).

It is also useful to interpret these definitions on the class of chordal bigraphs. Recall that bigraphs are bipartite, 
and hence automatically irreflexive. Also recall, that to see them as a special case of strongly chordal digraphs
(and a special case of strongly chordal graphs with possible loops) we view their adjacency matrix as first listing 
the red vertices and then the blue vertices. (This way independent permutations of each set of coloured vertices
correspond to simultaneous permutations of the vertices.) A strong ordering of $G$ corresponds to an ordering 
of the red vertices and an ordering of the blue vertices so that for red $v_i, v_j$ and blue $v_k, v_{\ell}$ we have 
$i < j, k < \ell$, and $v_iv_k \in E(G)$, $v_iv_{\ell} \in E(G),$ $v_jv_k \in E(G)$ imply $v_jv_{\ell} \in E(G)$. A bigraph 
has a strong ordering if and only if it has a simple ordering \cite{hmp}. A bigraph is chordal if and only if it does not 
contain an induced even cycle of length greater than $4$ \cite{golumbicJGT2}.

We prove the following extension of a result of Farber \cite{farberDM43}, who proved it for reflexive graphs. 
We will show in later sections that such results do not hold for digraphs, or even tournaments.

\begin{theorem}\label{sym}
Let $G$ be a graph with possible loops. The following statements are equivalent:
\begin{enumerate}
\item $G$ is strongly chordal;
\item $M(G)$ is totally balanced;
\item every induced subgraph of $G$ has a simple vertex;
\item $G$ has a simple ordering.
\qed
\end{enumerate}
\end{theorem} 

\begin{proof} If $G$ is strongly chordal, it has a strong, and hence a simple ordering. Consider the bigraph $B(G)$
obtained from $G$ by replacing each vertex $v$ by two vertices $v_1, v_2$, and each edge $vw$ by the two edges 
$v_1w_2, w_1,v_2$. It is easy to see that $B(G)$ also has a simple ordering, whence the bi-adjacency matrix $N(G)$
is totally balanced. Since $M(G)=N(B(G))$, this implies that $M(G)$ is totally balanced. Thus $1$ implies $2$.

To show $2$ implies $3$, suppose that $M(G)$ is totally balanced. Since every induced subgraph of a
strongly chordal graph is obviously strongly chordal, it suffices to show that $G$ has a simple vertex. Since
$M(G)=N(B(G))$ is totally balanced, $B(G)$ is a chordal bigraph and hence has a simple vertex $v_1$ or $v_2$ 
for some vertex $v$ of $G$, whence $v$ is a simple vertex in $G$.

We will now show that $3$ implies $1$. So assume that every induced subgraph of $G$ has a simple vertex.
We show how to obtain a strong ordering $v_1, v_2, \dots,v_n$ of $G$. The selection of $v_i$ for each $i \geq 1$
will be guided by a partial order $\preceq_i$ defined on $V_i = V(G) \setminus \{v_1, v_2, \dots, v_{i-1}\}$.
Initially, $V_0 = V(G)$ and $\preceq_0$ on $V_0$ consists of the reflexive pairs only, that is, $x \preceq_0 y$ if and
only if $x = y$ for all $x, y \in V_0$. For each $i \geq 1$, let $\preceq_i$ on $V_i$ be defined by $x \preceq_i y$ if and 
only if $x \preceq_{i-1} y$ or $N_i(x) \subset N_i(y)$ where $N_i(x)$ and $N_i(y)$ are the neighbourhoods
of $x$ and $y$ in the subgraph of $G$ induced by $V_i$. Equivalently, for each $i \geq 1$, $x \preceq_i y$ if and
only if $x = y$ or $x \neq y$ and $N_j(x) \subset N_j(y)$ for some $j \leq i$. We will show that $\preceq_i$ is a
partial order for each $i \geq 0$. The vertex $v_i$ for each $i \geq 1$ is selected to be a simple vertex that is
also a minimal element in the poset $(V_i,\preceq_i)$. We will also show that such a vertex $v_i$ always exists.

First we prove that $\preceq_i$ is a partial order on $V_i$ for each $i \geq 0$ by induction. Clearly, $\preceq_0$ is a
partial order on $V_0$. Assume that $i \geq 1$ and $\preceq_j$ is a partial order for each $j < i$. The reflexivity of 
$\preceq_i$ follows from the fact that $\preceq_i$ contains $\preceq_0$, which is reflexive. Suppose that $x \preceq_i y$ 
where $x \neq y$. Then there exists $j$ with $j \leq i$ such that $N_j(x) \subset N_j(y)$. Thus $N_j(y) \not\subset N_j(x)$ 
for all $j \leq i$, i.e., $y \not\preceq_i x$. Hence $\preceq_i$ is antisymmetric. For the transitivity, suppose that 
$x \preceq_i y \preceq_i z$. Then there exist $j, k$ with $j \leq i$ and $k \leq i$ such that $N_j(x) \subset N_j(y)$ and 
$N_k(y) \subset N_k(z)$. Let $\ell = \mbox{max}\{j,k\}$. Then $\ell \leq i$ and $N_\ell(x) \subset N_\ell(z)$, which means 
that $x \preceq_i z$. Therefore $\preceq_i$ is a partial order on $V_i$ for each $i \geq 0$.

Let $u$ be a simple vertex in the subgraph of $G$ induced by $V_i$. Such a vertex exists because every
induced subgraph of $G$ has a simple vertex. We prove that if $v \preceq_i u$ then $v$ is also a simple vertex.
So suppose that $v \preceq_i u$. Then there exists a $j$ with $j \leq i$ such that $N_j(v) \subset N_j(u)$. Hence
we must have $N_i(v) \subseteq N_i(u)$. Since $u$ is simple, $v$ is also simple. It follows that the subgraph
of $G$ induced by $V_i$ has a simple vertex that is also a minimal element in the poset $(V_i,\preceq_i)$ for each $i
\geq 1$. Therefore we obtain an ordering $v_1, v_2, \dots, v_n$. It suffices to show that the ordering is a strong
ordering of $G$.

Suppose that $i < j$, $k < \ell$, $v_i \in N(v_k), v_i \in N(v_\ell)$ and $v_j \in N(v_k)$. We show that $v_j \in N(v_\ell)$. 
By symmetry, we may assume that $i \leq k$. Thus $v_i, v_j, v_k, v_\ell \in V_i$. Since $v_i$ is simple, either 
$N_i(v_k) \subseteq N_i(v_\ell)$ or $N_i (v_\ell) \subset N_i(v_k)$. In the latter case, $v_\ell \preceq_i v_k$, and hence 
$v_\ell \preceq_k v_k$. However, $v_k$ is minimal in $(V_k,\preceq_k)$ by the choice of $v_k$, a contradiction. 
Therefore $N_i(v_k) \subseteq N_i(v_\ell)$, which implies $v_j \in N(v_\ell)$. This shows that the ordering 
$v_1, v_2, \dots, v_n$ is a strong ordering of $G$.

Finally we note that statements 3 and 4 are obviously equivalent.
\end{proof}

Let $W:\ v_0v_1 \dots v_k$ be a walk of length $k$ in a graph $G$ with possible loops. If $v_0 = v_k$, then
$W$ is called a {\em closed} walk. A {\em subwalk} of a walk $W$ is a walk $v_iv_{i+1} \dots v_j$ for some 
$0 \leq i \leq j \leq k$. A subwalk of $W$ is {\em proper} if the length of the subwalk is less than the length of 
$W$. A {\em strong chord of a walk} $W$ is an edge $v_iv_j$ (possibly a loop if $v_i = v_j$) such that $j-i$ is 
odd but not equal to $1$ or $-1$. A {\em strong chord of a closed walk} $W$ is defined similarly, except the 
expression $j-i$ is evaluated modulo $k$. 
Note that in a closed walk $W : v_0, v_1, \dots, v_k=v_0$, the last edge $v_{k-1}v_0$ is not an strong chord for any 
$k$ because $0$ is $(k-1)+1$ modulo $k$.

\begin{corollary}
\label{cor:chss-wlk}
A graph $G$ with possible loops is strongly chordal if and only every even closed walk  of length at least $6$ 
has a strong chord.
\end{corollary}

\begin{proof}
This follows from Theorem~\ref{sym} and the fact $M(G)$ is  totally balanced if and only if in $G$ every even 
closed walk of length at least $6$ has a strong chord \cite{ansteeJA5}.
\end{proof}

A matrix reformulation of this result states the following: {\em a symmetric $0, 1$ matrix which has a
$\Gamma$-free ordering also has a symmetric $\Gamma$-free ordering}. For matrices with $1$'s on 
the main diagonal, this was proved in \cite{lubiwSIAMJC16}.

Corollary \ref{cor:chss-wlk} characterizes strongly chordal graphs with possible loops by means of a forbidden 
structure, namely, even closed walks without strong chords. Recall that in the reflexive case, a characterization 
is also known in terms of forbidden induced subgraphs, namely cycles of length at least four, and trampolines 
\cite{farberDM43} (also called suns \cite{cn}). In the irreflexive case, it turns out that a characterization by
forbidden induced subgraphs is also known. First we note that all odd cycles are forbidden, because going 
around an odd cycle twice produces an even closed walk without strong chords. Thus, for irreflexive graphs 
only bipartite graphs can have a $\Gamma$-free ordering, and the characterization from \cite{golumbicJGT2} 
gives the forbidden induced subgraphs of chordal bigraphs, namely all even cycles of length greater than four. 
In conclusion, for irreflexive graphs, the forbidden induced subgraphs are all cycles of length different from four. 

Next we consider obstructions that are neither reflexive nor irreflexive.  Assume $C$ is a cycle with vertices
$0, 1, \dots, n-1, n > 4$, with one loop, at $0$; or two loops, at $0$ and $n-1$. A {\em regular fan at $0$} is 
the set of edges $0i$ for all even subscripts $i$. A {\em regular fan at $n-1$} is the set of edges $(n-1)(n-j)$ 
for all odd subscripts $j$. These concepts are illustrated in Figure \ref{F567}.

There is an infinite family of forbidden induced subgraphs consisting of even cycles with a loop and regular 
fan at $0$ as on the left side of Figure \ref{F567}, and two families of cycles with loops at $0, n-1$ -- one with 
a regular fan only at $0$, illustrated on he upper right of Figure \ref{F567}, and one with regular fans both at 
$0$ and at $n-1$, as on the lower right of Figure \ref{F567}.
It can be readily checked that each of these graphs contains an even closed 
walk of length at least $6$ without strong chords. Another infinite family of forbidden induced subgraphs 
consists of {\em weak trampolines}; these are obtained from reflexive trampolines by removing loops from
an arbitrary subset $S$ of the vertices of degree $2$, and adding an arbitrary set of disjoint edges between 
pairs of the vertices in $S$, cf. Figure \ref{F89}. (Note that this definition includes trampolines in the classical 
sense.) Finally, any path joining two reflexive vertices by a sequence of irreflexive vertices is also a forbidden 
induced subgraph, illustrated on the bottom of Figure \ref{F89}. In each of these graphs one can find an even 
closed walk of length at least $6$ without strong chords.

\begin{figure}[h]
\begin{center}
\includegraphics[height=4.4cm]{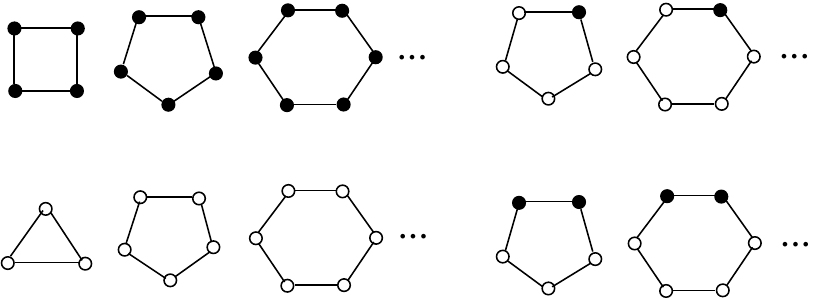}
\caption{Forbidden chordless cycles: families ${\cal F}_1, {\cal F}_2, {\cal F}_3, {\cal F}_4$} \label{F1234}
\end{center}
\end{figure}

\begin{figure}[h]
\begin{center}
\includegraphics[height=4.6cm]{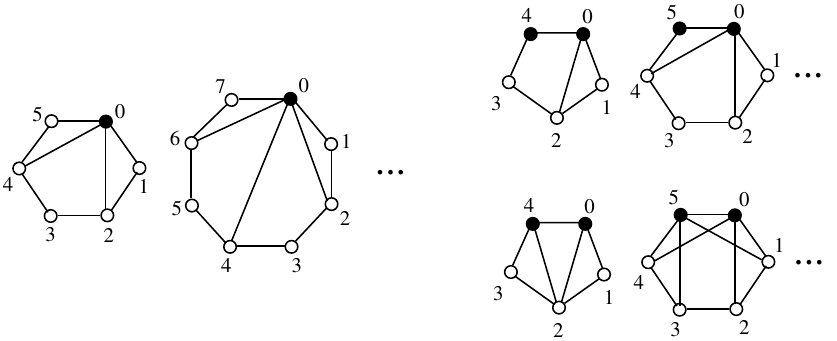}
\caption{Forbidden cycles with fans: families ${\cal F}_5, {\cal F}_6, {\cal F}_7$} \label{F567}
\end{center}
\end{figure}

\begin{figure}[h]
\begin{center}
\includegraphics[height=4.5cm]{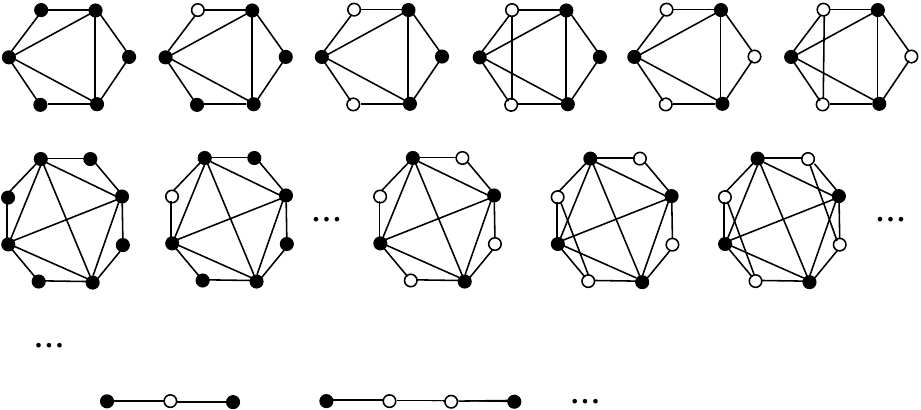}
\caption{Forbidden paths and weak trampolines: families ${\cal F}_8, {\cal F}_9$} \label{F89}
\end{center}
\end{figure}

The following more precise description characterizes strongly chordal graphs with possible loops by 
forbidden induced subgraphs.

\begin{theorem}\label{all}
A graph $G$ with possible loops is strongly chordal if and only if
it does not contain as an induced subgraph a graph in any of the 
following families:

\begin{enumerate}
  \item Family ${\cal F}_1$: reflexive cycles of length at least $4$;

  \item Family ${\cal F}_2$: irreflexive cycles of length other than $4$;

  \item Family ${\cal F}_3$: cycles of length at least $5$ with exactly
        one loop;

  \item Family ${\cal F}_4$: cycles of length at least $5$ with exactly
        two consecutive loops;

   \item Family ${\cal F}_5$: even cycles of length at least $6$ with a
        loop at $0$, with a regular fan at $0$;

  \item Family ${\cal F}_6$: cycles of length at least $5$ with two loops,
        at $0$ and $n-1$, with a regular fan at $0$;

  \item Family ${\cal F}_7$: cycles of length at least $5$ with two adjacent
        loops, at $0$ and $n-1$, and regular fans at both $0$ and $n-1$;

  \item Family ${\cal F}_8$: weak trampolines; and

  \item Family ${\cal F}_9$: paths of length at least $2$ with two loops
        at the two end vertices.
\end{enumerate}
\end{theorem}

We write ${\cal F} = \bigcup_{i=1}^9 {\cal F}_i$.
As we noted, every graph in $\cal F$ contains an even closed walk of length at least $6$, 
without strong chords. Thus by Corollary \ref{cor:chss-wlk}, we conclude that any graph with
possible loops that contains a graph from ${\cal F}$ as an induced subgraph is not strongly 
chordal. We prove the converse of this statement is also true.

The following two lemmas describe the cases of reflexive and irreflexive graphs, and follow
from known results on chordal bipartite graphs \cite{golumbic} and strongly chordal graphs 
\cite{farberDM43}, as discussed above.

\begin{lemma} \label{sub}
If the subgraph of $G$ induced by reflexive vertices is not a chordal graph, then
$G$ contains a graph in ${\cal F}_1$ as an induced subgraph. If the subgraph of
$G$ induced by irreflexive vertices is not a chordal bigraph, then $G$ contains
a graph in ${\cal F}_2$ as an induced subgraph.
\qed
\end{lemma}

\begin{lemma} \label{farber}
Suppose that $H$ is a chordal graph on vertices $v_0, v_1, \dots, v_{k-1}$ where 
$k \geq 6$ is even. If $N(v_i) =  \{v_{i-1}, v_{i+1}\}$ for each even $i$ and the
vertices with even subscripts form an independent set, then $H$ contains a
graph in ${\cal F}_8$ as an induced subgraph. In particular, if a chordal 
graph contains an even closed walk of length at least $6$ without strong chords, 
then it contains a graph from ${\cal F}_8$ as an induced subgraph.
\qed
\end{lemma}

Suppose $W: v_0, v_1, \dots, v_j=v_0, v_{j+1}, \dots, v_k$ is any walk,
and consider its closed proper subwalk $W': v_0, v_1, \dots, v_j (=v_0)$. If $j$ is even, then
the edge $v_jv_{j+1}$  is a strong chord, and if $j$ is odd and $v_0$ has a loop, then that 
loop is a strong chord. For future use, we formalize these observations as follows.

\begin{lemma} \label{subwalk}
Suppose $W: v_0, v_1, \dots, v_j=v_0, v_{j+1}, \dots, v_k$ is a walk without strong chords.
Then its closed proper subwalk $W': v_0, v_1, \dots, v_j (=v_0)$ has an odd length. Moreover,
the vertex $v_0$ must be irreflexive, unless $W=W'v_0$.
\qed
\end{lemma}

Note that in particular a closed walk of odd length that has no strong chords can only 
self-intersect if it uses a loop, i.e., if $W: v_0, v_1, \dots, v_j=v_0, v_{j+1}, \dots, v_k=v_0$ 
is a closed walk without strong chords, and $k$ is odd, then $j=1$ or $j+1=k$.

The next two auxiliary lemmas describe the possible shape of cycles with exactly one or two
loops in a graph with possible loops which does not contain an induced subgraph from ${\cal F}$.
They will be used repeatedly in our arguments.

\begin{lemma} \label{oneloop}
Let $G$ be a graph with possible loops and let $C: v_0v_1 \dots v_k,$ $k \geq 5,$ 
be a cycle in $G$, where $v_0=v_k$ is the only reflexive vertex of $C$. Suppose moreover, 
that the subpath $v_0v_1 \dots v_{k-1}$ of $C$ has no strong chords. Then $k$ is odd and 
$v_0v_j \in E(G)$ for all even $j < k$, or $G$ contains an induced subgraph from 
${\cal F}_2 \cup {\cal F}_3 \cup {\cal F}_5$.
\end{lemma}

\begin{proof}
We first note that the subpath $v_1v_2\dots v_{k-1}$ is an induced path; otherwise it
would have a chord, and since the chord isn't strong, $G$ would contain an induced
irreflexive cycle of length other than $4$, i.e., a graph from ${\cal F}_2$. Moreover, 
$v_0v_i \notin E(G)$ for each odd $i$ with $1 < i < k-1$, as these edges would be
strong chords. Suppose for contradiction that $v_0v_j \notin E(G)$ for some even $j$. 
Let $\ell$ be the greatest subscript with $1 \leq \ell < j$ such that $v_0v_{\ell} \in E(G)$ 
and $r$ be the least subscript with $j < r \leq k-1$ such that $v_0v_r \in E(G)$. Then 
$v_0, v_{\ell}, v_{\ell+1}, \dots, v_r$ induce a cycle in ${\cal F}_3$. Hence $v_0v_j \in E(G)$ 
for each even $j$; furthermore, $k$ must be odd, as otherwise $G$ would contain a cycle
in ${\cal F}_5$.
\end{proof}

\begin{lemma} \label{twoloops}
Suppose $G$ is a graph with possible loops and $W: v_0v_1 \dots v_k$ is a closed walk in 
$G$, of odd length $k > 3$, such that the subwalk $v_0v_1 \dots v_{k-1}$ has no strong chords.
Suppose moreover that $v_0$ and $v_{k-1}$ are the only reflexive vertices in $W$. Then 
$G$ contains an induced subgraph from $\bigcup_{i=2}^7 {\cal F}_i$.
\end{lemma}

Recall that a strong chord in a walk not viewed as closed is an edge $v_iv_j$ with $j-i$ odd
and not equal to $1, -1$, where the difference is not computed modulo $k$.

\begin{proof}
We first observe that we may assume that $W$ is a cycle; indeed Lemma \ref{subwalk}
specifies any proper subwalk would be odd, which is not possible for an odd walk. (Note
that the last option $W=vW'$ would imply that $k$ is even, so it cannot occur.)
 
Since $W$ has no strong chords, we must have $v_0v_i \notin E(G)$ for each odd 
$i$ with $1 < i < k$ and $v_iv_{k-1} \notin E(G)$ for each odd $i$ with $0 < i < k-2$. 
Let $\ell$ be the greatest subscript with $1 \leq \ell < k-1$ such that $v_0v_{\ell} \in E(G)$,
and let $r$ be least subscript with $0 < r \leq k-2$ such that $v_rv_{k-1} \in E(G)$. 
If $\ell < r$, then $v_0, v_{\ell}, v_{\ell+1}, \dots, v_r, v_{k-1}$ induce a graph in ${\cal F}_4$. 
So assume that $\ell \geq r$. Note that $\ell$ and $r$ are both even. Applying Lemma 
\ref{oneloop} to the cycle $v_0v_1 \dots v_{\ell}v_0$ we may conclude that $v_1v_j \in E(G)$ 
for each even $j$ with $0 < j \leq \ell$ (else $G$ contains an induced subgraph from 
${\cal F}_2 \cup {\cal F}_3 \cup {\cal F}_5$). Similarly, applying Lemma \ref{oneloop} to 
the cycle $v_{k-1}v_{k-2} \dots v_rv_{k-1}$ we may conclude that $v_jv_{k-1} \in E(G)$ for 
each even $j$ with $r \leq j < k-1$. If $\ell = k-3$ and $r = 2$, then $v_0, v_1, \dots, v_{k-1}$ 
would induce a graph in ${\cal F}_7$ as an induced subgraph. So we must have
$\ell < k-3$ or $r > 2$, and so $v_\ell, v_{\ell+1}, \dots, v_{k-1}, v_0$ or
$v_{k-1}, v_0, v_1, \dots, v_r$ induce a graph in ${\cal F}_6$.
\end{proof}

\begin{lemma} \label{adjloops}
Suppose that $u, v$ are non-adjacent reflexive vertices in $G$. If there is an
induced $(u,v)$-path whose internal vertices are not all reflexive, then $G$
contains a graph ${\cal F}_9$ as an induced subgraph. In particular, if there is
a $(u,v)$-walk whose internal vertices are all irreflexive, then $G$
contains a graph ${\cal F}_9$ as an induced subgraph.
\qed
\end{lemma}

We are now ready to prove the missing direction for Theorem \ref{all}.

\begin{lemma} \label{finisher}
If $G$ is not strongly chordal, then it contains a graph in ${\cal F}$ as an 
induced subgraph.
\end{lemma}

\begin{proof}
Suppose that $G$ is not strongly chordal. By Corollary \ref{cor:chss-wlk}, $G$
contains an even closed walk $W: v_0v_1 \dots v_k$, of length at least $6$, 
without strong chords. Consider first the case when $W$ is not a cycle (i.e., 
$W$ contains a repeated vertex). 

Suppose first that there is a vertex which appears twice consecutively in $W$,
say $v_0=v_1$. (Thus $v_0$ is a reflexive vertex and the loop $v_0v_0$ is an 
edge of $W$.) Since $W$ has no strong chords, $v_0v_i = v_1v_i \notin E(G)$ 
for any $2 < i < k-1$. We claim that $v_i \neq v_0$ for any $1 < i < k$. Indeed, 
if $v_i = v_0$, then $v_1v_2 \dots v_i$ is a proper closed walk of $W$. This
contradicts Lemma \ref{subwalk}, as $v_1=v_i$ is a reflexive vertex. The subwalk 
$W': v_1v_2 \dots v_k$ is a closed walk of an odd length. If $v_0$ is the only reflexive 
vertex in $W$ and $W'$ is a cycle, then by Lemma \ref{oneloop} (applied to $W'$) 
$G$ contains a graph in ${\cal F}$ as an induced subgraph (since we have shown 
that $v_0v_i = v_1v_i \notin E(G)$). 

If $v_0$ is the only reflexive vertex in $W$ and $W'$ is not a cycle, then $W'$ contains 
a cycle not containing $v_0$, which implies $G$ contains a graph in ${\cal F}_2$ as an 
induced subgraph. 

Suppose that $W$ has exactly two reflexive vertices. Let $v_a$ be the other 
reflexive vertex in $W$. If $a \notin \{2, k-1\}$, then $v_0v_a \notin E(G)$ as 
otherwise $W$ contains the strong chord $v_0v_a = v_1v_a$, a contradiction to 
the assumption that $W$ has no strong chords. Thus $v_0v_1 \dots v_a$ is a 
$(v_0,v_a)$-walk whose internal vertices are all irreflexive. By Lemma \ref{adjloops}, 
$G$ contains a graph in ${\cal F}_9$ as an induced subgraph. So $a \in \{2, k-1\}$. 
If $a = k-1$, then by Lemma \ref{twoloops} (applied to $v_1v_2 \dots v_k$) $G$ 
contains a graph in $\bigcup_{i=2}^7 {\cal F}_i$ as an induced subgraph. If $a = 2$, 
then again by Lemma \ref{twoloops} (applied to $v_1v_{k-1}v_{k-2} \dots v_1$) $G$ 
contains a graph in $\bigcup_{i=2}^7 {\cal F}_i$ as an induced subgraph.

Suppose now that $W$ has more than two reflexive vertices. A similar proof as
above shows that $v_2$ and $v_{k-1}$ are both reflexive. Since $W$ does not 
contain a strong chord, $v_2v_{k-1} \notin E(G)$. Let $v_2v_{j_1} \dots v_{j_t}v_{k-1}$ 
be the shortest $(v_2,v_{k-1})$-path in the subgraph of $G$ induced by $V(W-v_0)$. 
If some $v_{j_i}$ is irreflexive, then $G$ contains a graph in ${\cal F}_9$ as an 
induced subgraph according to Lemma \ref{adjloops}. On the other hand if each 
$v_{j_i}$ is reflexive, then $v_0, v_2, v_{j_1}, \dots, v_{j_t}, v_{k-1}$ induce a graph 
in ${\cal F}_1$. This completes the case when a vertex appears twice consecutively
in $W$, i.e., a loop is an edge of $W$.

Suppose next that $W$ contains a repeated vertex but no vertex appears twice 
consecutively in $W$. Then $W$ contains a closed proper walk. By Lemma \ref{subwalk} 
such a walk is of an odd length and no reflexive vertex can be a repeated vertex in $W$. 
Without loss of generality assume that $W': v_0v_1 \dots v_c$ is such a walk.
Then $c$ is odd and $v_0$ is irreflexive. Since $c$ is odd, $W'': v_cv_{c+1} \dots v_k$ 
is a closed proper subwalk of $W$ of an odd length. If $W'$ or $W''$ contains only 
irreflexive vertices, then $G$ contains an odd cycle consisting of irreflexive vertex and 
hence a graph in ${\cal F}_2$ as an induced subgraph. So we may assume that $W'$ 
and $W''$ both contain reflexive vertices. Let $v_f$ be the reflexive vertex in $W'$ with 
the greatest subscript and $v_g$ be the reflexive vertex in $W''$ with the least subscript. 
The choice of $v_f, v_g$ implies that the walk $v_fv_{f+1} \dots v_g$ whose internal 
vertices are all irreflexive. By Lemma \ref{adjloops}, $G$ contains a graph in ${\cal F}_9$ 
as an induced subgraph, or $v_fv_g \in E(G)$. So we may assume $v_fv_g \in E(G)$. 
Since $W$ does not contain a strong chord, $g-f$ is even. If $f \neq c-1$ or $g \neq c+1$ 
then $v_fv_{f+1} \dots v_gv_f$ is a closed walk of an odd length $> 3$ without strong 
chords, in which $v_f$ and $v_g$ are the only reflexive vertices. Applying Lemma \ref{twoloops} 
to this walk, we conclude that $G$ contains a graph in $\bigcup_{i=2}^7 {\cal F}_i$ as 
an induced subgraph. Hence we may assume that $f = c-1$ and $g = c+1$. Let $v_{f'}$ 
be the reflexive vertex in $W'$ with the least subscript (possibly $f' = f$) and $v_g$ be 
the reflexive vertex in $W''$ with the greatest subscript (possibly $g' = g$). By considering 
the walk $v_{g'}v_{g'+1} \dots v_kv_1 \dots v_{f'}$ and using a similar argument as 
for $v_f, v_g$, we may conclude that $v_{f'}v_{g'} \in E(G)$, and $f' = 1$, $g' = k-1$. 
Since no reflexive vertex can be a repeated vertex, $v_{f'} \neq v_g$. Since $W$ does 
not contain a strong chord, $v_{f'}v_g \notin E(G)$. Hence $v_{f'}, v_0, v_g$ induce
a graph in ${\cal F}_9$.  

Consider now the case when $W$ is a cycle. In view of Lemmas \ref{sub}, \ref{farber}, 
\ref{oneloop}, and \ref{twoloops}, we assume that $W$ contains an irreflexive vertex 
and at least three reflexive vertices. Suppose that $W$ contains consecutive irreflexive 
vertices. Without loss of generality assume that $v_1, v_2, \dots, v_{h-1}$ are 
irreflexive vertices where $h > 2$ and that $v_0$ and $v_h$ are reflexive.
If $v_0v_h \notin E(G)$ then $v_0v_1 \dots v_h$ is a walk connecting two 
reflexive vertices whose internal vertices are all irreflexive. 
By Lemma \ref{adjloops}, $G$ contains a graph in ${\cal F}_9$ as an induced
subgraph. So assume that $v_0v_h \in E(G)$. Then $h$ is even as $W$ has
no strong chords. Applying Lemma \ref{twoloops} to the cycle $v_0v_1 \dots v_hv_0$ 
we conclude that $G$ contains a graph in $\bigcup_{i=2}^7$ as an induced subgraph. 
Hence we assume that $W$ contains no consecutive irreflexive vertices. 

We prove by contradiction that any two irreflexive vertices are of an even distance 
from each other in $W$. So suppose that $v_r, v_s$ are two irreflexive vertices 
whose distance in $W$ is odd. Since the distance of $v_r, v_s$ in $W$ is odd, $r, s$ 
have different parity. Since no consecutive vertices in $W$ are irreflexive, the distance 
of $v_r, v_s$ is at least 3. Assume without loss of generality that $r$ is odd and 
$s$ is even. If $v_{r-1}v_{r+1} \notin E(G)$, then $\{v_{r-1}, v_r, v_{r+1}\}$ 
induce a graph in ${\cal F}_9$. So assume $v_{r-1}v_{r+1} \in E(G)$. 
Similarly, we may assume that $v_{s-1}v_{s+1} \in E(G)$.  
Let $P: v_{r+1}v_{\alpha_1} \dots v_{\alpha_p}v_{s-1}$ be an induced 
$(v_{r+1},v_{s-1})$-path in the subgraph of $G$ induced
by $v_{r+1}, v_{r+2}, \dots, v_{s-1}$ and
let $Q: v_{r-1}v_{\beta_1} \dots v_{\beta_q}v_{s+1}$ be an induced 
$(v_{r-1},v_{s+1})$-path in the subgraph of $G$ induced by 
$\{v_{r-1}, v_{r-2}, \dots, v_{s+1}\}$. If any of $P$ and $Q$ contains an 
internal irreflexive vertex, then $G$ contains a graph in ${\cal F}_9$ as
an induced subgraph by Lemma \ref{adjloops}. So assume that vertices in $P$
and $Q$ are all reflexive. The subgraph of $G$ induced by $V(P) \cup V(Q)$
contains the reflexive cycle $C$ formed by $P, Q,$ and the edges $v_{r-1}v_{r+1}$
and $v_{s-1}v_{s+1}$. Note that the subscripts $r-1, r+1$ are even and $s-1, s+1$ 
are odd; thus the path $P$ starts with a vertex with even subscript and ends with 
a vertex with odd subscript (and similarly for $Q$). Thus the path $P$ includes an 
edge $v_{\alpha_i}v_{\alpha_{i+1}}$ where the subscript $\alpha_i$ is even and 
the subscript $\alpha_{i+1}$ is odd. The only chords possible in $C$ are between 
a vertex in $P$ and a vertex in $Q$, as these paths are induced; moreover since 
there are no strong chords in $W$, the only chords possible in $C$ are between
vertices with subscripts of the same parity. Note that the edge $v_{\alpha_i}v_{\alpha_{i+1}}$
does belong to some cycle (e.g., $C$), but the shortest cycle it belongs to has length
greater than three, as the vertices of $Q$ with even subscripts are not adjacent to 
$v_{\alpha_{i+1}}$, and the vertices of $Q$ with odd subscripts are not adjacent to 
$v_{\alpha_i}$. Thus $C$ induces a graph that is not chordal, and contains a graph 
in ${\cal F}_1$ as an induced subgraph by Lemma \ref{sub}. Therefore any two 
irreflexive vertices are of an even distance from each other in $W$.

Since $W$ contains at least one irreflexive vertex, we may assume without loss 
of generality that $v_0$ is irreflexive. Then all irreflexive vertices in $W$ have 
even subscripts. Suppose that there is no edge between any two vertices of 
even subscript (i.e., the vertices of even subscripts form an independant set). 
If $v_i$ is irreflexive and $v_{i-1}v_{i+1} \notin E(G)$, then $v_{i-1}, v_i, v_{i+1}$ 
induce a graph in ${\cal F}_9$. Thus we assume that for each irreflexive $v_i$, 
the two neighbours $v_{i-1}, v_{i+1}$ of $v_i$ are adjacent. If the subgraph of 
$G$ induced by the reflexive vertices in $W$ is not chordal then $G$ contains 
a graph in ${\cal F}_1$ as an induced subgraph. On the other hand if the subgraph 
of $G$ induced by the reflexive vertices in $W$ is chordal, then the subgraph of 
$G$ induced by $V(W)$ is also chordal because each irreflexive vertex is simplicial 
in the subgraph. By Lemma \ref{farber}, the subgraph of $G$ induced by $V(W)$ 
contains a graph in ${\cal F}_8$ as an induced subgraph. 

It remains to consider the case when there are edges between vertices with even 
subscripts. If there is an edge between an irreflexive vertex $v_i$ and a reflexive 
vertex $v_j$ with even $j$, then $v_{i-1}, v_i, v_j$ induce a graph in ${\cal F}_9$,
since $W$ has no strong chords. 

Assume first that $v_bv_d$ is an edge between two reflexive vertices with even 
$b, d$, and assume that $0 < b < d$, and the difference $d-b$ is as large as 
possible. Consider any induced $(v_{d+1},v_{b-1})$-path $P'$ in $G$ contained 
in $\{v_{d+1}, v_{d+2}$, $\dots, v_{k-1}, v_0, v_1, \dots, v_{b-1}\}$. If any internal 
vertex of $P'$ is irreflexive then $G$ contains a graph from ${\cal F}_9$ as 
an induced subgraph, by Lemma \ref{adjloops}. Thus $P'$ together with the 
edges $v_{b-1}v_b, v_bv_d, v_dv_{d+1}$ forms a reflexive cycle containing
the edge $v_bv_d$. By Lemma \ref{sub} we can assume that the reflexive
vertices induce a chordal graph, and thus the edge $v_bv_d$ must belong to
a $3$-cycle with some vertex $v_c$. Since $b, d$ are even, $c$ must also
be even (else at least one of $v_bv_c, v_dv_c$ is a strong chord of $W$).
If $0 < c < b$ then $d - c > d - b$ which violates the choice of $v_b, v_d$.  
Similarly, if $d < c < k$ then also $c - b > d - b$  which also violates the 
choice of $v_b, v_d$. 

Finally, we consider edges between two vertices of even subscripts when these 
two vertices are both irreflexive. If such edges form a matching in $G$ then $G$ 
contains a graph ${\cal F}_8$ as an induced subgraph. If these edges don't form 
a matching, then one can verify that $G$ must contain a graph from 
${\cal F}_6 \cup {\cal F}_7$ as an induced subgraph.
\end{proof}

This completes the proof of Theorem \ref{all}.

\section{General Digraphs}

We now return to the context of strongly chordal digraphs, and review the relevant definitions first.

In \cite{haskinsSIAMJC2,meisterTCS463} the authors define a vertex $v$ in a digraph $D$ to be {\em simplicial}, 
if for all vertices $u \in N^-(v)$ and $w \in N^+(v)$, there is an arc $uw \in E(D)$. A {\em simplicial ordering} of a 
digraph $D$ is a linear ordering $v_1, v_2, \dots, v_n$ of its vertices, such that for each $i$, the vertex $v_i$ is 
simplicial in $D \setminus \{v_1, v_2, \dots, v_{i-1}\}$. A digraph is {\em chordal} if and only if it has a simplicial 
ordering.

We will call a vertex $v$ in a digraph $D$ {\em simple} if
\begin{itemize}
  \item $v$ is simplicial,
  \item if $x, y \in N^-(v)$, then $N^+(x) \subseteq
        N^+(y)$ or $N^+(x) \supseteq N^+(y)$, and
  \item if $x, y \in N^+(v)$, then $N^-(x) \subseteq
        N^-(y)$ or $N^-(x) \supseteq N^-(y)$.
\end{itemize}

A {\em simple ordering} of a digraph $D$ is a vertex ordering $v_1, v_2, \dots, v_n$ of $D$ such that for each $i$, 
the vertex $v_i$ is simple in $D \setminus \{v_1, v_2, \dots, v_{i-1}\}$. Observe that a simple ordering is again a 
simplicial ordering. A {\em strong ordering} of a digraph $D$ is a linear ordering $v_1, v_2, \dots, v_n$ of its vertices
such that for all $i < j$ and $k < \ell$ where $i, j, k, \ell$ are not necessarily distinct, if $v_iv_k \in E(D), v_iv_{\ell} \in E(D)$ 
and $v_jv_k \in E(D)$, then $v_jv_{\ell} \in E(D)$. A strong ordering of $D$ directly corresponds to a symmetric 
$\Gamma$-free ordering of $M(D)$. A strong ordering is a simple ordering and hence a simplicial ordering. A digraph 
is strongly chordal if and only if it has a strong ordering. Thus each strongly chordal digraph is a chordal digraph.

Having a simple ordering is equivalent to having a strong ordering in the classical context, but is not equivalent for 
general digraphs. Every symmetric $\Gamma$-free ordering is a simple ordering, but the converse is not necessarily
true. This is not true even for irreflexive tournaments; the irreflexive tournament $T_1$ in Figure~\ref{minobsirrtour} 
has a simple ordering (and its adjacency matrix is totally balanced), but is not strongly chordal, i.e., the matrix has
no symmetric $\Gamma$-free ordering. 

Every strongly chordal digraph has a simple
vertex. If $v$ is a reflexive simple vertex in
a strongly chordal digraph $D$, then $N^-(v)
\cup N^+(v)$ induces a semicomplete digraph in
$D$. It follows that the underlying graph of a
reflexive strongly chordal digraph is a chordal
graph.

\begin{lemma}
\label{lem:simple}
Let $D$ be a digraph.   If no vertex of $D$ is simple,
then $D$ is not strongly chordal.
\end{lemma}

\begin{proof}
As we noted in the previous paragraph, every strongly
chordal graph has a simple vertex, which is actually
the first vertex of the strong ordering.   Hence, if
no vertex of $D$ is simple, then no vertex of $D$ can
be the first in the strong ordering, and thus it is
not strongly chordal.
\end{proof}

A vertex $v$ in a digraph $D$ is a {\em peak} vertex
of $D$ if there exist vertices $u, w \in V(D)$ such that 
$uv \in E(D), vw \in E(D)$ and $uw \in E(D)$. 

\begin{lemma}\label{lem:peak}
An irreflexive vertex that is a peak cannot be the
last vertex in a simple ordering.
\end{lemma}

\begin{proof}
Let $v_1, \dots, v_n$ be a simple ordering of the vertices of
$D$, and assume $v_n$ is irreflexive and a peak vertex with
arcs $v_iv_j, v_iv_n, v_nv_j$ in $G$. Then a $\Gamma$
submatrix occurs in rows $i, n$ and columns $j, n$.
\end{proof}

\begin{corollary}
Let $D$ be an irreflexive digraph. If every vertex of
$D$ is a peak, then $D$ is not strongly chordal.
\end{corollary}

If follows from the proofs of Lemmas \ref{lem:simple}
and \ref{lem:peak} that in any strong ordering of a
strongly chordal digraph $D$ the first vertex must be
simple and the last vertex must not be an irreflexive peak.
Therefore, we observe for future reference that if an
irreflexive digraph $D$ has only one vertex that is
simple, and at the same time it is the only vertex of $D$
which is not a peak, then $D$ is not strongly chordal.

\begin{figure}[htb!]
\begin{center}
\begin{tikzpicture}
    \node [style=blackvertex] (1) at (0,0) {};
    \node [style=vertex] (2) at (0,1.5) {};
    \node [style=vertex] (3) at (1.5,1.5) {};
    \node [style=blackvertex] (4) at (1.5,0) {};
    \foreach \from/\to in {1/2,1/3,2/3,2/4,3/4,4/1}
        \draw [style=arc] (\from) to (\to);
\end{tikzpicture}
\end{center}
\vspace{-2mm}
\caption{Tournament $T_0$.}
\label{mixed_tournament}
\end{figure}
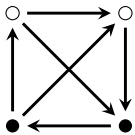

\begin{proposition}\label{jing}
If a digraph $D$ has a simple ordering, then $M(D)$ is totally balanced.
\end{proposition}

\begin{proof} Suppose a digraph $D$ has a simple ordering. A digraph $D$ defines a bigraph $B(D)$
(just as for graphs): each vertex $v \in V(D)$ gives rise to two vertices $v_1, v_2 \in V(B(D))$, and each
arc $vw \in E(D)$ gives rise to an edge $v_1w_2$ of $B(D)$. Then it is easy to see that the bigraph $B(D)$
also has a simple ordering. Thus the bi-adjacency matrix of $B(D)$ is totally balanced. Moreover, we have
$N(B(D))=M(D)$.
\end{proof}

The tournament $T_0$ in Figure~\ref{mixed_tournament} contains both reflexive and  irreflexive vertices.
It is not strongly chordal although each of the subgraphs induced by reflexive and irreflexive
vertices respectively is strongly chordal.

\section{Tournaments}

As we have seen, strongly chordal digraphs do 
not in general coincide with digraphs having a 
simple ordering, or having a totally balanced 
adjacency matrix, even for tournaments.   

We begin by addressing two
natural subcases, reflexive and irreflexive
tournaments.   Clearly, the matrix of a reflexive
directed cycle on three vertices is not totally
balanced (it is itself the bi-adjacency matrix of
an even cycle of length $6$). Thus,
every reflexive strongly chordal tournament is
acyclic, and we have the following theorem.

\begin{theorem}
\label{thm:reftour}
If $T$ is a reflexive tournament, then $T$ is strongly
chordal if and only if it is isomorphic to the reflexive
transitive tournament on $n$ vertices.
\end{theorem}

The irreflexive case, although more interesting, is
similar in flavor to the reflexive one.

For every integer $n$, $n \ge 3$, let $TT_n$ and
$TT^\ast_n$ denote the irreflexive transitive tournament
on $n$ vertices, and the tournament obtained from the
irreflexive transitive tournament on $n$ vertices where
the arc from the only source to the only sink has been
reversed.   It is easy to verify that ordering the
vertices of $TT_n$ increasingly with respect to
their in-degrees results in a $\Gamma$-free ordering;
the same order, up to reversing the arc from the first
to the last vertices, is a $\Gamma$-free ordering for
$TT^\ast_n$.   Hence, $TT_n$ and $TT^\ast_n$ are
strongly chordal digraphs for every $n \ge 3$.  For
integers $i, k, n$ such that  $2 \le n$, $0 \le i \le
n$ and $3 \le k$, we define $TT_n (i,k)$ to be the
tournament obtained from $TT_n$ by blowing up the
$i$-th vertex (in a transitive ordering) to a copy
of $TT^\ast_k$.

It is not hard to verify that the only strong
tournaments on three and four vertices are precisely
$TT^\ast_3$ and $TT^\ast_4$, and the only strong
tournament on five vertices which is also a strongly
chordal digraph is $TT^\ast_5$. The following lemma
generalizes these observations.

\begin{lemma} \label{irrstrtour}
Let $n$ be an integer, $n \ge 3$.   The only
irreflexive tournament on $n$ vertices which is both
strongly connected and strongly chordal is $TT^\ast_n$.
\end{lemma}

\begin{proof}
By induction on $n$.   We have already noticed that
the statement is true for $n \in \{ 3, 4, 5 \}$, so
let $n$ be at least $6$, and let $T$ be a strongly
connected and strongly chordal tournament.   Since
$T$ is strong, it is pancyclic, and hence it contains
a vertex $v$ such that the subtournament $T'$ obtained
from $T$ by deleting $v$ is strong.   Strong chordality
is a hereditary property and thus, by induction
hypothesis, $T'$ is isomorphic to $TT^\ast_{n-1}$.
Consider an ordering $v_1, v_2, \dots, v_{n-2}, v_{n-1}$
of $V(T')$ such that $v_{n-1} \to v_1$ and the reversal
of this arc results in the transitive tournament $TT_{n-1}$.

The following fact will be useful for the main argument
of the proof. Recall that the statement is true for $n
= 5$, and hence, every strongly connected subtournament
of $T$ on $5$ vertices should be isomorphic to $TT^\ast_5$.
The in-degree and out-degree sequences of $TT^\ast_5$ are
both $(3, 3, 2, 1, 1)$ (each in non-increasing order). In
order to obtain a contradiction, it will suffice to find
a strong subtournament of $T$ on $5$ vertices containing
at least four vertices with in- or out-degree at least $2$.

Since $T$ is strong, arcs from $T'$ to $v$ and from $v$
to $T'$ must exist in $T$.   We begin by showing that
$v_0 \to v$; suppose that $v \to v_0$ to reach a
contradiction.   Notice first that, if $d^-_T (v) \ge 2$,
then for any two integers $i, j$ with $1 < i < j < n-1$
such that $S = T[v_1, v_i, v_j, v_{n-1}, v]$ is a
subtournament of $T$ with $d^-_S (v) \ge 2$, we have that
$d^-_S(v_1), d^-_S(v_j), d^-_S(v_{n-1}), d^-_S(v) \ge 2$.
This, together with the fact that $S$ is strong, results
in a contradiction.   Thus, $d^-_T (v) = 1$.   Choose
integers $i$ and $j$ such that $1 < i < j < n-1$ and at
least one $v_i, v_j, v_{n-1}$ dominates $v$.   Again, let
$S$ be the induced subtournament $S = T[v_1, v_i, v_j,
v_{n-1}, v]$ of $T$, it is not hard to observe that, if
$v \to v_i$, then $S$ is a strong tournament on $5$
vertices with at least four vertices of in-degree at
least $2$, a contradiction. Hence, $v_i \to v$ and $v
\to \{ v_1, v_j, v_{n-1} \}$. Recall that $S$ is isomorphic
to $TT^\ast_5$, so it must contain an arc whose reversal
results in $TT_5$.  Only vertices $v$ and $v_i$ have
in-degree $1$ in $S$, so the only arcs that might have
this property are $(v_i, v)$ and $(v_1, v_i)$, but it is
routine to verify that none of them achieve the desired
result, a contradiction. Therefore $v_1 \to v$, and, an
analogous argument shows that $v \to v_{n-1}$.

We affirm that there exists $i \in \{ 2, \dots, n-2 \}$
such that $\{ v_1, \dots, v_i \} \to v \to \{ v_{i+1},
\dots, v_{n-1} \}$.   Suppose for a contradiction that
there are integers $i, j$ such that $1 < i < j < n-1$
and $\{ v_1, v_j \} \to v \to \{ v_i, v_{n-1} \}$.
Let $S$ be the induced subtournament $S = T[v_1, v_i,
v_j, v_{n-1}, v]$ of $T$.   Then, $d^+_S (v_1), d^+_S
(v_i), d^+_S (v_j), d^+_S (v) \ge 2$, a contradiction.
Hence, it is clear that reversing the arc $(v_{n-1},
v_0)$ in $T$ results in a transitive tournament, and
therefore $T$ is isomorphic to $TT^\ast_n$.
\end{proof}

Thus, in the strongly connected case, the only
strongly chordal irreflexive tournaments are very
close to a transitive tournament. As the following
argument shows, in the non-strong case, the
similarities are even more pronounced.

\begin{lemma}
Let $T$ be an irreflexive strongly chordal tournament.

If $T$ is non-strong, then $T$ is isomorphic to $TT_n (i,k)$ for some integers $i, k, n$
such that $1 \le i \le n$, $3 \le k$ and $2 \le n$.
\end{lemma}

\begin{proof}
It suffices to notice that, since $T_6$ (Figure
\ref{minobsirrtour}) is a minimal obstruction for
strong chordality and every strong tournament
contains a directed triangle, it is not possible for
two different strong components of $T$ to have more
than one vertex.
\end{proof}

We note that $TT^\ast_n = TT_1(1,n)$ and hence we can state both results together
as follows.

\begin{theorem}
Any irreflexive strongly chordal tournament is isomorphic to some $TT_n (i,k)$ with 
$1 \le i \le n$, and $3 \le k$.
\end{theorem}

In addition to the nice simple structure that
irreflexive strongly chordal tournaments have, it
is possible to characterize them by a small set of
minimal forbidden induced subgraphs.

It is a tedious, yet straightforward process to check
that all the strong tournaments on $5$ vertices, except
for $TT^\ast_5$, are minimal obstructions for strong
chordality, and the tournament $T_6$ obtained by
taking two disjoint copies of the directed $3$-cycle
and adding all the arcs from one to the other is also
a minimal obstruction for strong chordality.

Let $\cal{T}$ be the family of tournaments $\{ T_1,
\dots, T_6 \}$ depicted in Figure \ref{minobsirrtour}.
By applying Lemma \ref{lem:peak} to tournaments $T_3,
\dots, T_6$ it is easy to conclude, after a simple
exploration, that these tournaments are not strongly
chordal.   Similarly, using Lemma \ref{lem:simple} on
$T_2$, we conclude that it is not strongly chordal.
As for $T_1$, a simple exploration shows that there
is only one vertex which is not a peak, and at the
same time it is the only vertex which is simple.
Therefore, as observed after the proof of Lemma
\ref{lem:peak}, $T_1$ is not strongly chordal. Since
every irreflexive tournament on $4$ vertices is
strongly chordal, we conclude that tournaments in
$\cal{T}$ are minimal digraph obstructions for
strong chordality in the family of tournaments.

\begin{theorem}
\label{thm:irrtouGfree}
If $T$ is an irreflexive tournament then $T$ is strongly
chordal if and only if it is $\cal{T}$-free.
\end{theorem}

\begin{proof}
As we have already observed in the previous paragraph,
tournaments in $\cal{T}$ are minimal digraph
obstructions for strong chordality. We will show that
in the family of tournaments, these are all.  Notice
that if a tournament $T$ has a $\Gamma$-free ordering,
then we can add a sink or
a source, and still have a $\Gamma$-free ordering, it
suffices to add the new vertex at the end of the ordering.
Thus, tournament minimal obstructions for strong chordality
have neither sinks nor sources.

Let $T$ be a tournament which is a minimal obstruction
for strong chordality. Assume first that $T$ contains
a non-peak vertex, say $v$.   Then, by the definition
of peak, we obtain that $N^+ (v) \to N^- (v)$.   If
$|N^+(v)|, |N^-(v)| \ge 2$, then $v$ together with any
two vertices in $N^+(v)$ and any two vertices in
$N^-(v)$ induce a copy of $T_1$.   The minimality of
$T$ implies that $T$ is isomorphic to $T_1$.   Else,
either $|N^+(v)| = 1$ or $|N^-(v)| = 1$, we will assume
without loss of generality the former case. Since all
tournaments on four or less vertices admit a
$\Gamma$-free ordering, it must be the case that
$|N^-(v)| \ge 3$.   If $T[N^+(v)]$ contains a directed
triangle, then it is easy to find an induced $T_2$
in $T$ using the vertices of such triangle, $v$ and
the only vertex in $N^+(v)$.   Else, $T[N^+(v)]$ is
a transitive tournament, but in this case it is easy
to verify that $T$ is isomorphic to $TT^\ast_n$,
contradicting that $T$ is not strongly chordal.

Thus, we may assume that every vertex in $T$ is a
peak vertex.   A simple computational search shows
that there are no tournaments with this property
on less than $5$ vertices, that the only tournaments
on $5$ vertices where every vertex is a peak vertex
are $T_3$, $T_4$ and $T_5$, and the only tournament
on six vertices with this property is $T_6$.   So,
we may assume that $T$ has at least $7$ vertices.
Let $v$ be a vertex in $T$, and consider $T' =
T - v$.   Since $T'$ is strongly chordal, then
there is at least one non-peak vertex, say $u$, in
$T'$.   As in the previous case, if $|N_{T'}^+(u)|,
|N_{T'}^-(u)| \ge 2$, then we can find a copy of
$T_1$ in $T$, contradicting the choice of $T$.
Thus, either $d_{T'}^-(u) \le 1$ or $d_{T'}^+(u)
\le 1$.   Let us assume without loss of generality
the latter.  Consider first the case where $d_{T'}^+
(u) = 1$ and let $w$ be the only out-neighbour of
$u$ in $T'$.   Then, also as in the previous case,
if there is a directed triangle in $N_{T'}^-(u)$,
then we can find a copy of $T_2$ as an induced
subgraph of $T$, contradicting the choice of
$T$.   Thus, $N_{T'}^-(u)$ must induce a transitive
tournament. Since $u$ is a non-peak vertex, we have
that $w \to N_{T'}^-(u)$, and hence, $T'$ is
isomorphic to $TT^\ast_{n-1}$.

Recall that in $T$, $u$ is a peak vertex in $T$
and, since $N_{T'}^+(w) = N_{T'}^-(u)$, then
either $u \to v$ and there is a vertex $x$ such
that $x \to u$ and $x \to v$, or $v \to u$ and
$v \to w$.   In the latter case, since $v$ is
a peak vertex in $T$, there must exist a vertex
$x$ in $N_{T'}^-(u)$ such that $x \to v$. Also,
since $w$ is a peak vertex in $T$, there must
be a vertex $y$ in $N_{T'}^-(u)$ such that $v
\to y$.   We have that $\{ u, v, w, x, y \}$
induces a tournament on five vertices such
that every vertex is a peak vertex, i.e., one
of the tournaments $T_3$, $T_4$ or $T_5$,
contradicting the choice of $T$.   In the
former case, either $v \to w$, and in order
for $w$ to be a peak vertex in $T$ there is
a vertex $x$ in $N_{T'}^-(u)$ such that $v \to
x$, or $w \to v$, and in order for $v$ to be
a peak vertex in $T$ there is a vertex $x$ in
$N_{T'}^-(u)$ such that $v \to x$.   In either
case, if $y$ is any vertex in $N_{T'}^-(u)$
different from $x$, the set $\{ u, v, w, x, y
\}$ induces one of the tournaments $T_3$,
$T_4$ or $T_5$, contradicting again the choice
of $T$.   This closes the case where $d^+_{T'}(u)
= 1$.

So, let us assume that $d_{T'}^+(u) = 0$. Notice
that $T$ has neither sinks nor sources, and thus,
$u \to v$.   Since $v$ is a peak vertex in $T$,
there are vertices $w$ and $x$ in $T$ such that
$x \to v$, $v \to w$ and $x \to w$.   Since $u$
is dominated by every vertex in $T'$, we have
that $x \ne u$ and since $v \to w$, then $w \ne
u$.   Thus, $u, v, w$ and $x$ are four different
vertices, $(u, v, w, u)$ is a directed cycle in
$T$, and $x \to \{ u, v, w \}$.   If every vertex
in $T$ dominates $\{ u, v, w \}$, then, the
digraph $T_1$ induced by $V_T - \{ u, v, w \}$ in
$T$ should be acyclic, otherwise $T$ would contain
a copy of $T_6$ as an induced subgraph.   But in
this case, an ordering of $V_T$ where $u, v$ and
$w$ are the first three vertices, and then the
vertices of $T_1$ are ordered in such a way that
their adjacency matrix is a lower triangular matrix,
is a $\Gamma$-free ordering of $T$, contradicting
the choice of $T$.   Thus, there is at least one
vertex in $T$, different from $u$ dominated by
$v$ or by $w$.   Consider the set $S$ of vertices
in $T$ that dominate $u, v$ and $w$; again, this
set must induce an acyclic subgraph of $T$, and
hence, there is a vertex of zero in-degree in this
induced subgraph; assume without loss of
generality that $x$ has this property.   But $x$
cannot be a source in $T$, and thus, there must
be a vertex $y$ dominating $x$, and by the choice
of $x$, we have that $v \to y$ or $w \to y$.   In
either case it is routine to verify that the
subgraph of $T$ induced by $\{ u, v, w, x, y \}$
is one of the tournaments $T_3$, $T_4$ or $T_5$,
contradicting the choice of $T$.

Since the cases are exhaustive, we conclude that
the only minimal obstructions for strong chordality
in the class of irreflexive tournaments are those
included in the family $\cal{T}$.
\end{proof}

\begin{center}
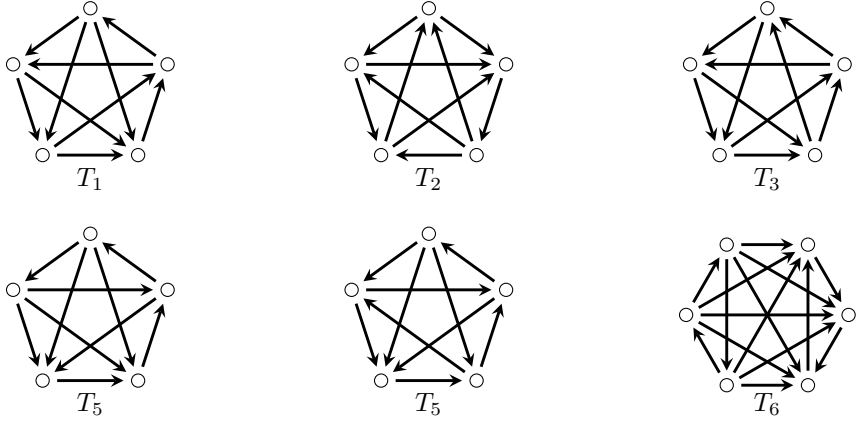
\begin{figure}[htb!]
\begin{tikzpicture}

\begin{scope}[scale=0.6]

\begin{scope}[xshift=-7.5cm,scale=0.9,rotate=90]
\foreach \i in {0,...,4}
	\draw ({(360/5)*\i}:2) node[style=vertex](\i){};
\foreach \i/\j in {0/1,0/2,0/3,1/2,1/3,2/3,2/4,3/4,4/0,4/1}
	\draw [style=arc] (\i) -- (\j);
\node (T1) at (-2.2,0){$T_1$};
\end{scope}

\begin{scope}[xshift=0cm,scale=0.9,rotate=90]
\foreach \i in {0,...,4}
	\draw ({(360/5)*\i}:2) node[style=vertex](\i){};
\foreach \i/\j in {0/1,0/4,1/2,1/4,2/0,2/4,3/0,3/1,3/2,4/3}
	\draw [style=arc] (\i) -- (\j);
\node (T2) at (-2.2,0){$T_2$};
\end{scope}

\begin{scope}[xshift=7.5cm,scale=0.9,rotate=90]
\foreach \i in {0,...,4}
	\draw ({(360/5)*\i}:2) node[style=vertex](\i){};
\foreach \i in {0,...,4}
	\draw [style=arc] let \n1={int(mod(\i+1,5))}
		in (\i) -- (\n1);
\foreach \i in {0,...,4}
	\draw [style=arc] let \n1={int(mod(\i+2,5))}
		in (\i) -- (\n1);
\node (T3) at (-2.2,0){$T_3$};
\end{scope}

\begin{scope}[xshift=-7.5cm,yshift=-5cm,scale=0.9,rotate=90]
\foreach \i in {0,...,4}
	\draw ({(360/5)*\i}:2) node[style=vertex](\i){};
\foreach \i/\j in {0/1,0/2,0/3,1/2,1/3,1/4,2/3,3/4,4/0,4/2}
	\draw [style=arc] (\i) -- (\j);
\node (T5) at (-2.2,0){$T_5$};
\end{scope}

\begin{scope}[xshift=0cm,yshift=-5cm,scale=0.9,rotate=90]
\foreach \i in {0,...,4}
	\draw ({(360/5)*\i}:2) node[style=vertex](\i){};
\foreach \i/\j in {0/1,0/2,0/3,1/2,1/4,2/3,3/1,3/4,4/0,4/2}
	\draw [style=arc] (\i) -- (\j);
\node (T5) at (-2.2,0){$T_5$};
\end{scope}

\begin{scope}[xshift=7.5cm,yshift=-5cm,scale=0.9]
\foreach \i in {0,...,5}
  \draw ({(360/6)*\i}:2) node[style=vertex](\i){};
\foreach \i/\j in {0/5,1/0,2/0,2/1,2/4,2/5,3/0,3/1,3/2,3/5,4/0,4/1,4/3,4/5,5/1}
  \draw [style=arc] (\i) -- (\j);
\node (T6) at (0,-2.2){$T_6$};
\end{scope}

\end{scope}

\end{tikzpicture}
\caption{The family $\cal{T}$.}
\label{minobsirrtour}
\end{figure}
\end{center}

We conclude by allowing loops to be present or
absent. In a tournament $T$ with possible loops,
we say a set of vertices is {\em acyclic} if in $T$
it contains no directed cycle (other than a loop).

The following lemma can be verified by a lengthy
but straightforward calculation.

\begin{lemma}
\label{lem:irr-gen}
Let $T$ be a tournament obtained from a tournament
in the family $\cal{T}$ by adding loops to an
acyclic set of vertices, and such that the resulting
tournament does not contain $T_0$ (from
Figure \ref{mixed_tournament}) as a subgraph.   Then
$T$ is a minimal obstruction for strong chordality.
\end{lemma}

Lemma \ref{lem:irr-gen} will be used multiple times
in the proof of our following theorem. 

\begin{theorem}
Any strongly chordal tournament $T$ with possible loops
is obtained from $TT_n (i,k)$, for some integers $i, k, n$
with $1 \le i \le n$, $3 \le k$, by adding loops to vertices
in any acyclic subset of vertices, as long as their addition 
does not create a copy of $T_0$ from Figure \ref{mixed_tournament}.
\end{theorem}

\begin{proof}
It is a simple excercise to verify that the tournament $T_0$
from Figure \ref{mixed_tournament}, is a minimal 
obstruction for strong chordality.

Now, when $T$ is strong, notice that either the same
$\Gamma$-free ordering used for $TT^\ast_n$ in the
irreflexive case, or its reverse, will also work for
this case.  The only case, up to symmetry, where the
order needs to be reversed, is when the vertices are
ordered by decreasing out-degree, and the first
vertex has a loop.   Also, if $n \ge 4$, then it can
never happen that the first and last vertices are
reflexive, otherwise $T$ would contain $T_0$.

Notice that adding a source or a sink to a
$\Gamma$-free tournament will result again in a
$\Gamma$-free tournament, regardless of whether the
new vertex is reflexive or irreflexive.   To obtain
a $\Gamma$-free ordering for the new tournament, it
suffices to add the new vertex at the end of the
previous ordering. Thus, indeed the tournaments
described in the theorem are strongly chordal.

Let $T$ be a strong tournament which is strongly
chordal.   If the underlying irreflexive tournament
$T^\circ$ of $T$ is isomorphic to $TT^\ast_n$, then
$T$ does not contain $T_0$ as an induced subgraph,
and it has the desired form.   Else, by Theorem
\ref{thm:irrtouGfree}, $T^\circ$ contains $T_i$ as a
subtournament, for $i \in \{ 1, \dots, 5 \}$.    If
$T_i$ is also a tournament of $T$, then $T$ is not
strongly chordal, a contradiction.   Thus, $T$
contains a copy of $T_i$ where some vertices are
reflexive.   But this is not possible either,
because the directed reflexive triangle is a minimal
obstruction for strong chordality, as well as $T_0$
and $T_i$ with any acyclic subset of vertices being
reflexive, and not containing $T_0$. Thus, $T$ must
have the structure described in the first item of
the theorem.

Now, if $T$ is non-strong, then every strong
component of $T$ is either a single vertex or
contains a directed triangle.   Since the
reflexive $3$-cycle and each tournament
obtained from $T_6$ by adding loops to an
arbitrary acyclic subset are minimal
obstructions for strong chordality, it follows
that at most one connected component is not a
single vertex.   Hence, the only non-trivial
strong component of $T$ has the structure
described by the first item of this theorem,
and thus, $T$ has the desired structure.
\end{proof}

\section{Conclusions}

We have seen that strongly chordal digraphs can be recognized
in polynomial time amongst symmetric digraphs, and amongst
tournaments with possible loops. We do not know if they can be
recognized in polynomial time in general. We now mention one
other natural class of digraphs with polynomial recognition of
strong chordality.

Each bipartite graph $G$ defines a digraph $D_G$ by orienting all edges 
from red to blue vertices; the adjacency matrix of $D_G$ is clearly obtained 
from the bi-adjacency matrix of $G$ by adding rows and columns of zeros.
Thus independent permutations of rows and columns of $N(G)$ again yield 
a symmetric ordering of $M(D_G)$. This means that $G$ is a chordal bigraph 
if and only if $D_G$ is a strongly chordal digraph.

A {\em balanced digraph} is a digraph $D$ such that any cycle
has the same number of forward and backward arcs. By definition, 
a balanced digraph $D$ is irreflexive, and it is easy to see that there 
is a vertex partition into parts $V_i, i=1, 2, \dots, k$, such that each
arc of $D$ starts in some $V_i$ and ends in $V_{i+1}$. The
adjacency matrix of a balanced digraph has can be symmetrically
permuted into consecutive blocks corresponding to the parts $V_i$.
In such a form, a symmetric permutation of the matrix corresponds
to independent permutations of rows and columns in each submatrix
$M_i$ with rows in block $V_i$ and columns in block $V_{i+1}$.
Moreover, it is easy to see that each $\Gamma$ submatrix of $M$
must lie in some $M_i$. Note that when $k=2$, i.e., when there are
only two parts, $V_1, V_2$, a balanced digraph is some $D_G$ for
a bipartite graph $G$. For a general balanced digraph, denote by
$G_i$ the underlying bipartite subgraph of $D$ with parts $V_i, V_{i+1}$.

\begin{theorem}
A balanced digraph $D$ is strongly chordal if and only if each $G_i$
is a chordal bigraph.
\end{theorem}

We can translate this result to a forbidden subgraph characterization.
A {\em fence} is an oriented even cycle of length greater than four,
without a directed path of length two, see Figure \ref{fff}.

\begin{corollary}
A balanced digraph $D$ is strongly chordal if and only if it does not
contain a fence as an induced subgraph.
\end{corollary}

\begin{figure}[hh]
\begin{center}
\includegraphics[height=1.5cm]{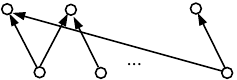}
\caption{A fence} \label{fff}
\end{center}
\end{figure}

\begin{corollary}
Each oriented tree is strongly chordal.
\end{corollary}

\vspace{2mm}

To close the paper, we note that problems for graphs with possible loops are often very natural
and combine the reflexive and irreflexive versions in interesting ways. Consider for example the
problems of domination and total domination. The {\em domination number} of a graph $G$ is the
minimum number of vertices in a {\em dominating set} $D$, i.e., a set such that each vertex is in 
$D$ or has a neighbour in $D$. The {\em total domination number} of a graph $G$ is the minimum 
number of vertices in a {\em total dominating set} $D$, i.e., a set such that each vertex has a 
neighbour in $D$. This suggests that the former deals with graphs that are reflexive, as each
vertex dominates itself (as if it had a loop), while the latter deals with graphs that are irreflexive,
no vertex is adjacent to itself, so it cannot dominate itself. We can more generally define the {\em 
general domination number of a graph $G$ with possible loops} to be minimum number of vertices 
in a set $D$ such that each vertex has a neighbour in $D$. If $G$ is reflexive, the general domination
number coincides with the usual domination number, and if $G$ is irreflexive, it coincides with the total 
domination number. For a general graph with possible loops, this new problem represents and interesting 
mixture of the two classical problems.

Farber \cite{domdom} gave a linear time algorithm to find a minimum dominating set in a reflexive strongly 
chordal graph, and Damaschke, Mueller, and Kratsch \cite{damdom} gave a linear time 
algorithm for a minimum total dominating set in a bipartite chordal graph. We now describe a linear time
algorithm for the problem to find a minimum general dominating set in a strongly chordal graph with 
possible loops, which generalizes and extends both the algorithms in \cite{domdom,damdom}. It is, in 
fact, a very small modification of the algorithm of Farber \cite{domdom}, underlining the fact of just how 
natural the new class of strongly chordal graphs with possible loops is.

As in \cite{domdom}, we consider the minimum general dominating set problem together with its dual, the maximum
number of vertices with disjoint neighbourhoods. While in Farber's case, the neighbourhoods were all closed
neighbourhoods, for us they are open neighbourhoods, which may or may not include the vertex itself, depending
on whether or not it has a loop. Clearly a vertex can be dominated only by a vertex from its neighbourhood, so vertices
with disjoint neighbourhoods need to be dominated by distinct vertices. Therefore, the maximum number of vertices 
with disjoint neighbourhoods is a lower bound for the minimum number of vertices in a dominating set. Moreover, if 
we find a dominating set $D$ and a set $C$ of vertices with disjoint neighbourhoods such that $|C|=|D|$, then $D$ 
is minimum and $C$ is maximum.

These sets will be computed iteratively in the order of a strong ordering $<$ of a graph $G$ with possible loops.
We repeat the following steps, each assigning to some vertices of $G$ labels $C, D$, and $N$. (The label $N$
signals that the label $C$ is no longer available for that vertex.) We assume that $G$ has no isolated vertices, as 
such vertices can always be dealt with separately. Initially, no vertices are labeled.

\begin{itemize}
\item
Find, in the ordering $<$, the first vertex $x$ without the label $N$.
\item
Find, in the ordering $<$, the last neighbour $y$ of $x$.
\item
Label $x$ by $C$, label $y$ by $D$, and label all neighbours of $y$ by $N$.
\end{itemize}

Note that a vertex will in general receive several labels. Every vertex will receive at least the label $C$ or $N$.
Moreover, every vertex will receive the label $C$ and $D$ at most once. Specifically, when a vertex $x$ is labeled 
$C$, a unique neighbour $y$ of $v$ is labeled $D$. At that point, all neighbours of $y$ receive the label $N$, 
including $x$. Therefore $x$ will not receive the label $C$ again, and $y$ will never be receiving another label
$D$ (since all its neighbours are ineligible for label $C$). So, if we ignore the auxiliary labels $N$, we will have
some $k$ vertices labeled $C$ and the same number $k$ vertices labeled $D$. (Some vertices may have both
labels $C$ and $D$.) In other words, we have sets $C$ and $D$ (of vertices with those labels) that have the same
cardinality. The final set $D$ is dominating, as there are no vertices left without a label $C$ or $N$, and each vertex
with label $C$ or $N$ has a neighbour labeled $D$. We now prove that the neighbourhoods of vertices labeled $C$ 
are disjoint. Otherwise some $x < x'$ both labeled $C$ have a common neighbour $z$; suppose $y$ was the last 
neighbour of $x$ when $x$ was labeled $C$. Since $y$ is the last neighbour of $x$, we have $z < y$. Since $x'$ is
labeled $C$ later than $x$, it is not a neighbour of $y$. Therefore we have $x < x', z < y$, and $xz \in E(G), xy \in E(G)$, 
and $x'z \in E(G), x'y \notin E(G)$, which contradicts the fact that $<$ is a strong ordering. Thus we have a general 
dominating set $D$ and a set of vertices $C$ with disjoint neighbourhoods, and $|C|=|D|$. Therefore both are optimal. 

We have given a linear time algorithm solving the general domination problem (and its dual) if a strong ordering of the
graph with possible loops is given. We expect that the algorithm for weighted domination in strongly chordal graphs 
\cite{domdom} also allows a similar extension to weighted general domination in strongly chordal graphs with possible 
loops.

\end{document}